\documentclass[final]{elsarticle}
\usepackage[top=2.5cm, bottom=2.5cm, left=3cm, right=3cm]{geometry}

\makeatletter
\def\ps@pprintTitle{%
 \let\@oddhead\@empty
 \let\@evenhead\@empty
 \def\@oddfoot{\centerline{\thepage}}%
 \let\@evenfoot\@oddfoot}
\makeatother

\usepackage{hyperref}
\usepackage{graphicx}
\usepackage{amsfonts}
\usepackage{amssymb}
\usepackage{amsmath}
\usepackage{amsthm}
\usepackage{float}
\usepackage{exscale}
\usepackage{relsize}
\usepackage{bm}
\usepackage{subeqnarray}
\usepackage{fancyhdr}
\usepackage{lastpage}
\usepackage{layout}
\usepackage{cleveref}
\usepackage{appendix}
\usepackage{bbm}
 
\newtheorem{thm}{Theorem}[section]

\newtheorem{proposition}[thm]{Proposition}
\newtheorem{lemma}[thm]{Lemma}
\newtheorem{definition}[thm]{Definition}

\newcommand{\bx}{\bm{x}}
\newcommand{\by}{\bm{y}}
\newcommand{\trace}{\mathrm{Tr}}

\newcommand{\da}{\downarrow}
\newcommand{\ua}{\uparrow}

\begin{document}

\begin{frontmatter}

\title{\textbf{Generality of Lieb's Concavity Theorem}}
\author{De Huang\fnref{myfootnote} }
\address{Applied and Computational Mathematics, California Institute of Technology, Pasadena, CA 91125, USA}
\fntext[myfootnote]{E-mail address:\ dhuang@caltech.edu.}
 
\begin{abstract}
We show that Lieb's concavity theorem holds more generally for any unitarily invariant matrix function $\phi:\mathbf{H}^n_+\rightarrow \mathbb{R}$ that is monotone and concave. Concretely, we prove the joint concavity of the function $(A,B) \mapsto\phi\big[(B^\frac{qs}{2}K^*A^{ps}KB^\frac{qs}{2})^{\frac{1}{s}}\big] $ on $\mathbf{H}_+^m\times\mathbf{H}_+^n$, for any $K\in \mathbb{C}^{m\times n},s\in(0,1],p,q\in[0,1], p+q\leq 1$. This result improves a recent work by Huang \cite{huang2019generalizing} for a more specific class of $\phi$.
\end{abstract}

\begin{keyword}
Lieb's concavity theorem, matrix functions, symmetric forms, majorization. 
\vspace{2mm}
\MSC[2010] 47A63, 15A42, 15A16
\end{keyword}

\end{frontmatter}

\section{Introduction}
Lieb's Concavity Theorem \cite{LIEB1973267}, as one of the most celebrated results in the study of trace inequalities, states that the function 
\begin{equation}\label{eqt:LCT}
(A,B)\ \longmapsto\ \trace[K^*A^pKB^q]
\end{equation}
is jointly concave on $\mathbf{H}_+^m\times\mathbf{H}_+^n$, for any $K\in \mathbb{C}^{m\times n}$, $p,q\in(0,1], p+q\leq 1$. Here $\mathbf{H}_+^n$ is the convex cone of all $n \times n$ Hermitian, positive semidefinite matrices. Among rich consequences of the Lieb's concavity theorem, the concavity of the map $A\mapsto \trace[\exp(H+\log A)]$ on $\mathbf{H}_+^n$ and the three-matrix extension of the Golden-Thompson inequality, both also established by Lieb \cite{LIEB1973267}, are most acknowledged. As an important application, Lieb and Ruskai \cite{doi:10.1063/1.1666274} used these results to prove the strong subadditivity of quantum entropy. 

Since its original establishment, the concavity of \eqref{eqt:LCT} has been discussed from various perspectives and proved alternatively using, for example, the theory of Herglotz functions (Epstein \cite{epstein1973remarks}), quadratic interpolations (Uhlmann \cite{uhlmann1977relative}, Kosaki \cite{kosaki1982interpolation}) and matrix tensors (Ando \cite{ANDO1979203}, Carlen \cite{carlen2010trace}, Nikoufar et al. \cite{NIKOUFAR2013531}). Recently, Huang \cite{huang2019generalizing} generalized Lieb's result to the concavity of 
\begin{equation}\label{eqt:generalLCT}
(A,B) \ \longmapsto\ \phi\big((B^\frac{qs}{2}K^*A^{ps}KB^\frac{qs}{2})^{\frac{1}{s}}\big), 
\end{equation}
on $\mathbf{H}_+^m\times\mathbf{H}_+^n$ for any $K\in \mathbb{C}^{m\times n}$, $s\in(0,1],p,q\in(0,1],p+q\leq 1$, with $\phi(\cdot)=(\trace_k[\cdot])^\frac{1}{k},1\leq k\leq n$.
Here the $k$-trace $\trace_k(A)$ of a matrix $A\in \mathbb{C}^{n\times n}$ is defined as 
\[\trace_k(A) = \sum_{1\leq i_1<i_2<\cdots<i_k\leq m} \lambda_{i_1}\lambda_{i_2}\cdots \lambda_{i_k},\quad 1\leq k\leq n,\]
with $\bm{\lambda}(A)=(\lambda_1,\lambda_2,\dots,\lambda_n)$ being the eigenvalues of $A$, counting multiplicities. Huang's proof was a direct use of an operator interpolation technique by Stein \cite{stein1956interpolation}, and hence also provided a new proof of the original Lieb's concavity theorem. An application of Huang's on $k$-trace generalization is to derive concentration estimates on partial spectral sums of random matrices \cite{2018arXiv180805550H}, which extended Tropp's master bounds \cite{Tropp2012,MAL-048} from the largest (or smallest) eigenvalue to the sum of the $k$ largest (or smallest) eigenvalues. Later, Huang \cite{huang2019improvement} strengthened his result by showing that the map \eqref{eqt:generalLCT} is jointly concave for arbitrary $\phi:\mathbf{H}^n_+\rightarrow \mathbb{R}$ that is unitary invariant, monotone(monotone increasing with respect to L\"owner order), concave and satisfies H\"older's inequality, i.e. $\phi(|AB|)\leq \phi(|A|^p)^\frac{1}{p}\phi(|B|^q)^\frac{1}{q},\forall p,q\in[1,+\infty],\frac{1}{p}+\frac{1}{q}=1$. 

However, though the immediate operator interpolation arguments in Huang's previous proof require $\phi$ to satisfy H\"older's inequality, the final result actually does not. In this paper, we will further improve Huang's results by removing the H\"older's condition. More precisely, we will prove the concavity of \eqref{eqt:generalLCT} for \textbf{arbitrary $\phi:\mathbf{H}^n_+\rightarrow \mathbb{R}$ that is unitary invariant, monotone and concave}. We remark that for any symmetric function $\phi:\mathbb{R}^n_+\rightarrow \mathbb{R}$ that is monotone increasing (with respect to the standard vector partial order) and concave, its extension to $\mathbf{H}^n_+$ defined as $\phi(A)=\phi(\bm{\lambda}(A)),A\in \mathbf{H}^n_+$ is unitarily invariant, monotone and concave on $\mathbf{H}^n_+$. 

The proof of our further generalization will be based on an observation that, given any function $\mathcal{F}:\Omega\rightarrow \mathbf{H}^n$ from a convex set $\Omega$ to the space of all Hermitian matrices $\mathbf{H}^n$, 
\begin{align*}
&\text{$X\mapsto\phi(\mathcal{F}(X))$ is concave on $\Omega$ for arbitrary $\phi:\mathbf{H}^n_+\rightarrow \mathbb{R}$}\\
&\text{that is unitary invariant, monotone and concave $\mathbf{H}^n$,}
\end{align*} 
if and only if 
\[\text{$X\mapsto\sum_{i=1}^k\lambda^\ua_i(\mathcal{F}(X))$ is concave on $\Omega$ for all $1\leq k\leq n$,}\]
where $\lambda^\ua_i(A)$ denotes the $i_{\text{th}}$ smallest eigenvalue of $A\in\mathbf{H}^n$. That is to say, we only need to prove the concavity of $\eqref{eqt:generalLCT}$ for $\phi(X)=\sum_{i=1}^k\lambda^\ua_i(\mathcal{F}(X)),1\leq k\leq n$. This strategy shares the spirit of Ky Fan's dominance theorem (e.g. see Theorem 7.4.8.4 in \cite{horn2012matrix}): given any $A,B\in \mathbb{C}^{n\times n}$, $\|A\|\leq \|B\|$ for arbitrary unitarily invariant norm $\|\cdot\|$ if and only if the singular values of $A$ is weakly majorized by the singular values of $B$. A similar idea was adopted in a recent work by Hiai et al. \cite{hiai2017generalized}, in which they used majorization theories to show that, to prove a class of integral inequality for arbitrary unitarily invariant norm of Hermitian matrices requires only to prove it for every Ky Fan $k$-norm(sum of the $k$-largest eigenvalues). Inspired by their work, we will also use techniques of majorization to prove our preceding observation. Then we will prove the concavity of $\eqref{eqt:generalLCT}$ for $\phi(X)=\sum_{i=1}^k\lambda^\ua_i(\mathcal{F}(X)),1\leq k\leq n$ based on a new variational form of the sum the $k$ smallest eigenvalues: for any $f:\mathbb{R}_+\rightarrow\mathbb{R}$ that is monotone increasing and satisfies $f(0)=0$, we have
\[\sum_{i=1}^k \lambda^\ua_i\big(f(M^*AM)\big) 
= \inf_{\begin{subarray}{c}G\in\mathbb{C}^{n\times n},G^2=G\\\mathrm{rank}(G)=k\end{subarray}} \trace\big[f(M^*G^*AGM)\big],\quad \text{for any}\ A\in \mathbf{H}_+^n,M\in \mathbb{C}^{n\times n}.\]

\subsection*{outline}
The rest of the paper is organized as follows. \Cref{sec:Notations&MainResults} is devoted to introductions of general notations, the notion of symmetric forms and our main results. We will briefly review in \Cref{sec:Majorization} the theories of majorization and use them to prove a useful equivalence theorem. The proofs of our main theorems are presented in \Cref{sec:Proofs}. 

\section{Notations and Main Results}
\label{sec:Notations&MainResults}

\subsection{General conventions}
For any positive integers $n,m$, we write $\mathbb{C}^n$ for the $n$-dimensional complex vector spaces equipped with the standard $l_2$ inner products, and $\mathbb{C}^{m\times n}$ for the space of all complex matrices of size $m\times n$. Let $\mathbb{R}^n,\mathbb{R}_+^n,\mathbb{R}_{++}^n$ be $(-\infty,+\infty)^n,[0,+\infty)^n,(0,+\infty)^n$, respectively. Let $\mathbf{H}^n$ be the space of all $n\times n$ Hermitian matrices, $\mathbf{H}_+^n$ be the convex cone of all $n\times n$ Hermitian, positive semi-definite matrices, and $\mathbf{H}_{++}^n$ be the convex cone of all $n\times n$ Hermitian, positive definite matrices. We write $I_n$ for the identity matrix of size $n \times n$. We use $S_n$ to denote the symmetric group of all permutations of order $n$. 

For any $\bx=(x_1,\dots,x_n),\by=(y_1,\dots,y_n)\in \mathbb{R}^n$, we write $\bx+\by$ and $\bx\by$ for the entry-wise sum and entry-wise product respectively, i.e.
\[\bx+\by=(x_1+y_1,\dots,x_n+y_n),\quad \bx\by = (x_1y_1,\dots,x_ny_n).\]
We say $\bx\leq \by$ if $x_i\leq y_i,i=1,\dots,n$. We will denote by $\bx^\da$ and $\bx^\ua$ the descending reordering and ascending reordering of $x$, respectively. That is, there exist some permutations $P_1,P_2\in S_n$ such that $\bx^\da = P_1\bx,\bx^\ua=P_2\bx$, and 
\[x^\da_1\geq x^\da_2 \geq \cdots\geq x^\da_n,\quad x^\ua_1\leq x^\ua_2 \leq \cdots\leq x^\ua_n.\]
For any function scalar function $f:\mathbb{R}\rightarrow\mathbb{R}$, the extension of $f$ to a function from $\mathbb{R}^n$ to $\mathbb{R}^n$ is given by
\[f(\bx) = (f(x_1),\dots,f(x_n)),\quad x\in \mathbb{R}^n. \]

For any $A\in \mathbf{H}^n$, we use $\lambda_1(A),\lambda_2(A),\dots,\lambda_n(A)$ to denote all the eigenvalues of $A$ and write $\bm{\lambda}(A) = (\lambda_1(A),\dots,\lambda_n(A))\in\mathbb{R}^n$. We will be frequently using $\bm{\lambda}^\da(A)$ and $\bm{\lambda}^\ua(A)$ as the descending ordering and ascending ordering, respectively, of the eigenvalues of $A$, i.e. $\lambda^\da_i(A)$ is the $i_{\text{th}}$ largest eigenvalue of $A$, and $\lambda^\ua_i(A)$ is the $i_{\text{th}}$ smallest eigenvalue of $A$. For any scalar function $f:\mathbb{R}\rightarrow\mathbb{R}$, the extension of $f$ to a function from $\mathbf{H}^n$ to $\mathbf{H}^n$ is given by 
\[f(A)=\sum_{i=1}^nf(\lambda_i(A))u_iu_i^*, \quad A\in \mathbf{H}^n,\]
where $u_1,u_2,\cdots,u_n\in\mathbb{C}^n$ are the corresponding normalized eigenvectors of $A$. Then obviously, the spectrum of $f(A)$ is $f(\bm{\lambda}(A))$; and if $f$ is monotone increasing on $\mathbb{R}$, then $\lambda^\da_i(f(A)) = f(\lambda^\da_i(A))$. One can find more discussions and analysis on matrix functions in \cite{carlen2010trace,Vershynina:2013}.

\subsection{Symmetric forms}
We start with symmetric functions on $\mathbb{R}^n$ defined as follows.
\begin{definition}\label{def:SymmetricForm}
A function $\phi:\mathbb{R}^n\rightarrow \mathbb{R}$ is a \textbf{symmetric form} if it is invariant under permutation:
\[\text{$\phi(\bx) = \phi(P\bx)$ for any $\bx\in \mathbb{R}^n_+$ and any permutation $P\in S_n$.}\]
A symmetric form $\phi$ is \textbf{monotone} (increasing) if 
\[ \text{$\bx\geq \by$ implies $\phi(\bx)\geq \phi(\by)$, for any $\bx,\by\in\mathbb{R}^n$.}\]
A symmetric form $\phi$ is \textbf{convex}, if 
\[\text{$\phi(\tau \bx+(1-\tau)\by)\leq \tau \phi(\bx)+(1-\tau)\phi(\by)$, for any $\bx,\by\in \mathbb{R}^n_+$ and any $\tau\in[0,1]$.}\]
A symmetric form $\phi$ is \textbf{concave} if $-\phi$ is convex.
\end{definition}

The domain of a symmetric form $\phi$ can be naturally extended from $\mathbb{R}^n$ to $\mathbf{H}^n$, by feeding $\phi$ the eigenvalues of a matrix in $\mathbf{H}^n$. 

\begin{definition}\label{def:MatrixForm}
The extension of a symmetric form $\phi$ to $\mathbf{H}^n$ is defined as
\[\phi(A) = \phi(\bm{\lambda}(A)),\quad A\in \mathbf{H}^n.\]
\end{definition}

\begin{proposition}\label{Prop:ExtensionProperties}
Let $\phi$ be a symmetric form on $\mathbb{R}^n$, then its extension to $\mathbf{H}^n$ is unitarily invariant: 
\[ \text{$\phi(U^*AU) = \phi(A)$, for any $A\in \mathbf{H}^n$ and any unitary matrix $U\in \mathbb{C}^{n\times n}$.}\]
If $\phi$ is monotone, then its extension to $\mathbf{H}^n$ is monotone with respect to L\"owner order:
\[\text{$A\succeq B$ implies $\phi(A)\geq\phi(B)$, for any $A,B\in\mathbf{H}^n$.}\] 
If $\phi$ is convex, then its extension to $\mathbf{H}^n$ is convex:
\[\text{$\phi(\tau A+(1-\tau)B)\leq \tau\phi(A)+(1-\tau)\phi(B)$, for any $A,B\in\mathbf{H}^n$ and any $\tau\in[0,1]$.} \] 
\end{proposition}

The unitary invariance and the monotonicity inheriting property follow straightforward from definition. The proof of the convexity inheriting property requires the use of majorization between eigenvalues. We hence divert the proof of \Cref{Prop:ExtensionProperties} to \Cref{sec:Majorization}. Due to the inheriting properties, in what follows we will not distinguish between a symmetric form and its extension to Hermitian matrices. We remark that, in many cases, the domain of a symmetric may be restricted to smaller regions that are permutatively invariant (e.g. $\mathbb{R}_+^n,\mathbb{R}_{++}^n$) or unitarily invariant (e.g. $\mathbf{H}_+^n,\mathbf{H}_{++}^n$), for effectiveness of monotonicity or convexity.

Generally, if a symmetric form $\phi$ is convex, homogeneous of order 1 and positive definite, i.e. 
\[\phi(\bx)=0\Longleftrightarrow \bx = (0,0,\dots,0),\]
then $\phi$ is called a symmetric gauge function. A famous bijection theory of von Neumann \cite{von1937some} says that any unitarily invariant matrix norm on $\mathbf{H}^n$ is the extension of some symmetric gauge function on $\mathbb{R}^n$. In this paper, however, our main results are most related to symmetric forms that are monotone and concave. Some examples of such class of symmetric forms are listed below.
\begin{enumerate}
\item The k-trace introduced in \cite{2018arXiv180805550H}:
\[\trace_k[\bx]^\frac{1}{k} = \left(\sum_{1\leq i_1<i_2<\cdots<i_k\leq n} x_{i_1}x_{i_2}\cdots x_{i_k}\right)^\frac{1}{k},\quad \bx\in\mathbb{R}_+^n,\quad 1\leq k\leq n.\]
\item The sum of rotated partial geometric means:
\[g_k(\bx) = \sum_{1\leq i_1<i_2<\cdots<i_k\leq n} (x_{i_1}x_{i_2}\cdots x_{i_k})^\frac{1}{k},\quad \bx\in\mathbb{R}_+^n,\quad 1\leq k\leq n.\]
\item The semi $p$-norm for $p\in(-\infty,0)\cup(0,1]$:
\[\|\bx\|_p = \left(\sum_{i=1}^n x_i^p\right)^\frac{1}{p},\quad \bx\in \mathbb{R}_+^n.\]
\item The weighted sum biased to smaller entries: given any $\bm{a}\in \mathbb{R}^n$,
\[\langle \bm{a}^\da,\bx^\ua\rangle =\sum_{i=1}^ka^\da_i x^\ua_i,\quad \bx\in \mathbb{R}^n.\]
In particular, the sum of the $k$ smallest entries:
\[\langle \mathbbm{1}_{\{i\leq k\}},\bx^\ua\rangle = \sum_{i=1}^k x^\ua_i,\quad \bx\in \mathbb{R}^n.\]
\end{enumerate}

Obviously, any positive combination of a collection of monotone, concave symmetric forms is still a monotone, concave symmetric forms. Also, we can generate many more monotone, concave symmetric forms by simply compositing with monotone, concave functions, as stated in the following proposition. 

\begin{proposition}\label{prop:Composition}
Let $\phi:\mathbb{R}^n\rightarrow\mathbb{R}$ be a symmetric form, and $f:\mathbb{R}\rightarrow\mathbb{R}$ be a function. 
\begin{itemize}
\item If $\phi$ is monotone, and $f$ is monotone increasing over $\mathrm{range}(\phi)$, then $f\circ \phi$ is a monotone symmetric form. If $\phi$ is convex, and $f$ is monotone increasing and convex over $\mathrm{conv}(\mathrm{range}(\phi))$, then $f\circ \phi$ is a convex symmetric form.
\item If $f$ is monotone increasing, and $\phi$ is monotone over $\mathrm{range}(f)^n$, then $\phi\circ f$ is a monotone symmetric form. If $f$ is convex, and $\phi$ is monotone and convex over $\mathrm{conv}(\mathrm{range}(f)^n)$, them $\phi\circ f$ is a convex symmetric form.
\end{itemize} 
\end{proposition}

Note that the trace function $\trace$ is a monotone, convex and also concave symmetric form on $\mathbf{H}^n$. Then combining \Cref{Prop:ExtensionProperties} and \Cref{prop:Composition}, we can conclude that for any monotone increasing function $f$ on $\mathbb{R}$, $\trace[f(\cdot)]$ is monotone on $\mathbf{H}^n$; and for any convex (or concave) function $f$ on $\mathbb{R}$, $\trace[f(\cdot)]$ is convex (or concave) on $\mathbf{H}^n$. Therefore we have provided an alternative proof for Theorem 2.10 in \cite{carlen2010trace}. 

\subsection{Main Results}
Our main purpose is to generalize Lieb's concavity theorems from trace to symmetric forms that are monotone and concave. Huang \cite{huang2019generalizing} applied operator interpolations to obtain generalizations of Lieb's concavity to k-traces $\phi(x) = \trace_k[x]^\frac{1}{k}$, which he used to derive concentration estimates on partial spectral sums of random matrices \cite{2018arXiv180805550H}. Since the interpolation part of Huang's proof requires essentially the symmetry and H\"older property of the $k$-trace, his results can be extended to more general symmetric forms that are monotone, concave and satisfies H\"older's inequality \cite{huang2019improvement}. Even further, we find the H\"older property actually unnecessary, and by adopting techniques of majorization we can strengthen Huang's results to the following.  

\begin{thm}[General Lieb's Concavity Theorem]\label{thm:GeneralLiebConcavity}
Let $\phi$ be a symmetric form that is monotone and concave on $\mathbb{R}_+^n$. Then for any $K\in \mathbb{C}^{m\times n}$ and any $s\in(0,1],p,q\in[0,1],p+q\leq 1$, the function 
\begin{equation}
(A,B) \ \longmapsto\ \phi\big((B^\frac{qs}{2}K^*A^{ps}KB^\frac{qs}{2})^{\frac{1}{s}}\big) 
\label{eqt:function1}
\end{equation}
is jointly concave on $\mathbf{H}_+^m\times\mathbf{H}_+^n$. 
\end{thm}

\begin{thm}\label{thm:GeneralLieb}
Let $\phi$ be a symmetric form that is monotone and concave on $\mathbb{R}_+^n$. Then for any $H\in \mathbf{H}^n$ and any $\{p_j\}_{j=1}^m\subset[0,1]$ such that $\sum_{j=1}^mp_j\leq1$, the function 
\begin{equation}
(A_1,A_2,\dots,A_m) \ \longmapsto\ \phi\big(\exp\big(H+\sum_{j=1}^mp_j\log A_j\big)\big)
\label{eqt:function2}
\end{equation}
is jointly concave on $(\mathbf{H}_{++}^n)^{\times m}$. In particular, $A\mapsto\phi\big(\exp(H+\log A)\big)$ is concave on $\mathbf{H}_{++}^n$.
\end{thm}

\Cref{thm:GeneralLiebConcavity} is a further generalization of the generalized Lieb's concavity theorem (Theorem 3.2) in \cite{huang2019generalizing} (see also Theorem 2.5 in \cite{huang2019improvement}), and \Cref{thm:GeneralLieb} is a further generalization of Theorem 3.3 in \cite{huang2019generalizing} (see also Corollary 6.1 in \cite{LIEB1973267} or Theorem 2.6 in \cite{huang2019improvement}). We will first show that it is sufficient to prove the concavity of $\eqref{eqt:function1}$ and \eqref{eqt:function2} with $\phi$ being the sum of the $k$ smallest eigenvalues for all $1\leq k\leq n$. This proof strategy is inspired by a recent work of Hiai et al. \cite{hiai2017generalized}, in which they used majorization theories to generalize some multivariate trace inequalities. They showed that, to prove a class of integral inequality for arbitrary unitarily invariant norm of Hermitian matrices requires only to prove it for every Ky Fan $k$-norm, namely the sum of the $k$ largest singular values. Following their idea, we will also use techniques of majorization to first obtain an equivalence theorem as follows. 

\begin{thm} \label{thm:GenericConcavity}
Let $\Omega$ be a convex set in some linear space, and $\mathcal{F}:\Omega \rightarrow \mathbf{H}^n$ be a function that maps $\Omega$ to $n\times n$ Hermitian matrices. Then the following two statements are equivalent:
\begin{itemize}
\item[(i)] For any monotone, convex symmetric form $\phi$ on $\mathbb{R}^n$, the map $X\mapsto \phi\big(\mathcal{F}(X)\big)$ is convex on $\Omega$.
\item[(ii)] For any $1\leq k\leq n$, the map $X\mapsto \sum_{i=1}^k\lambda^\da_i\big(\mathcal{F}(X)\big)$ is convex on $\Omega$.
\end{itemize}
Similarly, the following two statements are equivalent:
\begin{itemize}
\item[(i*)] For any monotone, concave symmetric form $\phi$ on $\mathbb{R}^n$, the map $X\mapsto \phi\big(\mathcal{F}(X)\big)$ is concave on $\Omega$.
\item[(ii*)] For any $1\leq k\leq n$, the map $X\mapsto \sum_{i=1}^k\lambda^\ua_i\big(\mathcal{F}(X)\big)$ is concave on $\Omega$.
\end{itemize}
\end{thm} 

We remark that, the convex part and the concave part of \Cref{thm:GenericConcavity} are equivalent. In fact, if the convex part is true, we can immediately prove the concave part by considering $\mathcal{F}(\cdot)\rightarrow -\mathcal{F}(\cdot)$, $\phi(\cdot)\rightarrow -\phi(-(\cdot))$ and noticing that $-\lambda^\da_i(\mathcal{F}(\cdot)) = \lambda^\ua_i(-\mathcal{F}(\cdot))$. The proof is diverted to the end of \Cref{sec:Majorization}, after our brief review on some fundamental theories of majorization. 

Supported by \Cref{thm:GenericConcavity}, we can confidently reduce our task to proving the concavity of \eqref{eqt:function1} and \eqref{eqt:function2} only for $\phi(\bx)=\sum_{i=1}^kx^\ua_i,1\leq k\leq n$. This will be done by interpreting the sum of the k smallest eigenvalues as the infimum of some specialized trace functions, using the following two lemmas. 

\begin{lemma}\label{lem:f00}
Let $f:\mathbb{R}_+\rightarrow \mathbb{R}$ be a monotone increasing function such that $f(0)=0$. Then for any $A\in\mathbf{H}_+^n,M\in \mathbb{C}^{n\times n}$ and any $k\leq n$,
\begin{equation}\label{eqt:f00}
\sum_{i=1}^k \lambda^\ua_i\big(f(M^*AM)\big) 
= \inf_{\begin{subarray}{c}G\in\mathbb{C}^{n\times n},G^2=G\\\mathrm{rank}(G)=k\end{subarray}} \trace\big[f(M^*G^*AGM)\big].
\end{equation}
Moreover, if $M$ is invertible, the infimum can be achieved.
\end{lemma}

\begin{lemma}\label{lem:fNI0}
Let $f:\mathbb{R}\rightarrow \mathbb{R}$ be a monotone increasing function such that $f(x\rightarrow-\infty)=0$. Then for any $A\in\mathbf{H}^n$ and any $k\leq n$,
\begin{equation}\label{eqt:fNI0}
\sum_{i=1}^k \lambda^\ua_i\big(f(A)\big) 
= \inf_{\begin{subarray}{c}H\in\mathbf{H}^n\\\mathrm{rank}(H)=n-k\end{subarray}} \trace\big[f(H+A)\big].
\end{equation}
\end{lemma}

The proofs of \Cref{lem:f00} and \Cref{lem:fNI0} will be presented in \Cref{sec:Proofs}, followed by the proofs of \Cref{thm:GeneralLiebConcavity} and \Cref{thm:GeneralLieb}.

We will be using frequently the following extended version of the Courant-Fisher characterization (min-max theorem) for eigenvalues of Hermitian matrices. One may refer to \cite{horn2012matrix,Parlett:1998:SEP:280490} for a proof.

\begin{thm}[Courant-Fisher] \label{thm:Courant-Fisher}
For any $A\in \mathbf{H}^n$ and any $0\leq m_1< m_2\leq n$, 
\begin{align}
\sum_{i=m_1+1}^{m_2} \lambda^\da_i(A) =&\  \max_{\begin{subarray}{c} U\in \mathbb{C}^{n\times m_2} \\ U^*U = I_{m_2}\end{subarray}} \min_{\begin{subarray}{c} V\in \mathbb{C}^{m_2\times (m_2-m_1)} \\ V^*V = I_{m_2-m_1}\end{subarray}} \trace[V^*U^*AUV] \\
=&\ \min_{\begin{subarray}{c} U\in \mathbb{C}^{n\times (n-m_1)} \\ U^*U = I_{n-m_1}\end{subarray}} \max_{\begin{subarray}{c} V\in \mathbb{C}^{(n-m_1)\times (m_2-m_1)} \\ V^*V = I_{m_2-m_1}\end{subarray}} \trace[V^*U^*AUV].\nonumber
\end{align}
\end{thm}

\section{Majorization} \label{sec:Majorization}
For any two vectors $\bm{a},\bm{b}\in \mathbb{R}^n$, $\bm{a}$ is said to be weakly majorized by $\bm{b}$, denoted by $\bm{a}\prec_w \bm{b}$, if 
\[\sum_{i=1}^ka^\da_i\leq \sum_{i=1}^kb^\da_i, \quad 1\leq k\leq n;\] 
moreover, a is said to be majorized by $b$, denoted by $\bm{a}\prec \bm{b}$, if equality holds for $k=n$, i.e.
\[\sum_{i=1}^na_i=\sum_{i=1}^nb_i.\] 
The following two lemmas are most important for deriving inequalities from majorization relations. One may refer to \cite{ando1989majorization,marshall1979inequalities,hiai2010matrix} for proofs and more discussions on this topic. 

\begin{lemma}\label{lem:MajorBridge}
For any $\bm{a},\bm{b}\in \mathbb{R}^n$, if $\bm{a}\prec_w \bm{b}$, then there is some $\bm{c} \in \mathbb{R}^n$ such that $\bm{a}\leq \bm{c}\prec \bm{b}$.
\end{lemma}

\begin{lemma}\label{lem:MajorKey}
For any $\bm{a},\bm{b}\in \mathbb{R}^n$, $\bm{a}\prec \bm{b}$ if and only if $\bm{a}=D\bm{b}$ for some doubly stochastic matrix $D$, i.e. $D_{ij}\geq 0, 1\leq i,j\leq n$, $\sum_{j=1}^nD_{ij} = 1, 1\leq i\leq n$ and $\sum_{i=1}^nD_{ij} = 1, 1\leq j\leq n$.
\end{lemma}

Since any doubly stochastic matrix is a convex combination of permutation matrices, an equivalent statement of \Cref{lem:MajorKey} is that, $\bm{a}\prec \bm{b}$ if and only if $\bm{a}$ is a convex combination of permutations of $\bm{b}$, i.e.
\[\bm{a} = \sum_{j=1}^m\tau_j P_j\bm{b},\]
for some $\{P_j\}_{j=1}^m\subset S_n$ and some $\{\tau_j\}_{j=1}^m\subset[0,1]$ such that $\sum_{j=1}^m\tau_j=1$.

\begin{lemma}\label{lem:MajorSymmetry}
Let $\phi:\mathbb{R}^n\rightarrow \mathbb{R}$ be a convex symmetric form. Then for any $\bm{a},\bm{b}\in \mathbb{R}^n$, $\bm{a}\prec \bm{b}$ implies $\phi(\bm{a})\leq \phi(\bm{b})$. Moreover, if $\phi$ is also monotone, then $\bm{a}\prec_w \bm{b}$ implies $\phi(\bm{a})\leq \phi(\bm{b})$. 
\end{lemma}

\begin{proof} If $\bm{a}\prec \bm{b}$, we have $\bm{a} = \sum_{j=1}^m\tau_j P_j\bm{b}$ for some permutations $\{P_j\}_{j=1}^m\subset S_n$ and some $\{\tau_j\}_{j=1}^m\subset[0,1]$ such that $\sum_{j=1}^m\tau_j=1$. Then using convexity and permutation invariance of $\phi$, we have
\[\phi(\bm{a}) = \phi(\sum_{j=1}^m\tau_j P_j\bm{b})\leq \sum_{j=1}^m\tau_j \phi(P_j\bm{b}) = \sum_{j=1}^m\tau_j \phi(\bm{b}) = \phi(\bm{b}). \]
If $\bm{a}\prec_w \bm{b}$, then by \Cref{lem:MajorBridge} there is some $\bm{c}\in \mathbb{R}^n$ such that $\bm{a}\leq \bm{c}\prec \bm{b}$. Then we have $\phi(\bm{c})\leq \phi(\bm{b})$. Moreover, if $\phi$ is monotone, we have $\phi(\bm{a})\leq \phi(\bm{c})$, and so $\phi(\bm{a})\leq \phi(\bm{b})$.  
\end{proof}

The following lemma is a widely used fact on majorization relations between eigenvalues. We provide the proof here to illustrate the proof techniques for majorization relations.

\begin{lemma}\label{lem:MatrixMajorization}
For any $A,B\in\mathbf{H}^n$, 
\begin{equation}\label{eqt:MatrixMajorization1}
\bm{\lambda}(A+B)\prec \bm{\lambda}(A)+\bm{\lambda}(B).
\end{equation}

\begin{proof} For any Hermitian matrix $A\in\mathbf{H}^n$ and any $1\leq k\leq n$, by \Cref{thm:Courant-Fisher} we have that 
\[\sum_{j=1}^k\lambda^\da_j(A) = \max_{\begin{subarray}{c} U\in \mathbb{C}^{n\times k}\\ U^*U = I_k \end{subarray}}\trace[U^*AU].\]
Therefore, for any $A,B\in\mathbf{H}^n$, we have
\begin{align*}
\sum_{j=1}^k\lambda^\da_j(A+B) =&\ \max_{\begin{subarray}{c} U\in \mathbb{C}^{n\times k}\\ U^*U = I_k \end{subarray}}\trace[U^*(A+B)U]\\
\leq&\ \max_{\begin{subarray}{c} U\in \mathbb{C}^{n\times k}\\ U^*U = I_k \end{subarray}}\trace[U^*AU] + \max_{\begin{subarray}{c} U\in \mathbb{C}^{n\times k}\\ U^*U = I_k \end{subarray}}\trace[U^*BU]\\
=&\ \sum_{j=1}^k\lambda^\da_j(A) + \sum_{j=1}^k\lambda^\da_j(B).
\end{align*}
And obviously we have 
\[\sum_{j=1}^n\lambda^\da_j(A+B)=\trace[A+B]=\trace[A]+\trace[B]= \sum_{j=1}^n\lambda^\da_j(A) + \sum_{j=1}^n\lambda^\da_j(B).\]
Therefore $\bm{\lambda}(A+B)\prec \bm{\lambda}(A)+\bm{\lambda}(B)$.
\end{proof}

\end{lemma}

\begin{proof}[\rm\textbf{Proof of \Cref{Prop:ExtensionProperties}}] 
$\phi(A)$ is only a symmetric function of the eigenvalues of $A$, and the eigenvalues of $A$ and $U^*AU$ are the same for any unitary $U$, we thus have $\phi(U^*AU)=\phi(A)$.

For any $A,B\in \mathbf{H}^n$, if $A\succeq B$, then $\bm{\lambda}^\da(A)\geq \bm{\lambda}^\da(B)$ by \Cref{thm:Courant-Fisher}. Therefore if $\phi$ is monotone as a vector symmetric form, we have $\phi(A)=\phi(\bm{\lambda}^\da(A)) \geq \phi(\bm{\lambda}^\da(B))=\phi(B)$.

For any $A,B\in \mathbf{H}^n$ and any $\tau\in [0,1]$, let $C=\tau A+(1-\tau) B$. By \Cref{lem:MajorKey}, we know that 
\[\bm{\lambda}(C) \prec \tau \bm{\lambda}(A)+(1-\tau)\bm{\lambda}(B).\]
Then by \Cref{lem:MajorSymmetry} and the convexity of $\phi$ as a vector symmetric form, we immediately have that 
\[\phi(C)=\phi(\bm{\lambda}(C))\leq \phi\big(\tau \bm{\lambda}(A)+(1-\tau)\bm{\lambda}(B)\big) \leq \tau\phi(A)+(1-\tau)\phi(B).\]
Therefore $\phi$ is also convex on $\mathbf{H}^n$.
\end{proof}

\begin{proof}[\rm\textbf{Proof of \Cref{thm:GenericConcavity}}] 
As we have mentioned, the convex part and the concave part of the theorem are equivalent. We thus only prove the equivalence between $(i\text{*})$ and $(ii\text{*})$. $(i\text{*})\Rightarrow (ii\text{*})$ is trivial, since $\phi_k(\bx) = \sum_{i=1}^kx^\ua_i$ is a monotone, concave symmetric form on $\mathbb{R}^n$ for each $1\leq k\leq n$. 

Now suppose $(ii\text{*})$ is true. Then for any $X,Y\in \Omega$ and any $\tau\in [0,1]$, with $Z=\tau X+(1-\tau)Y$, we have
\[\sum_{i=1}^k\lambda^\ua_i\big(\mathcal{F}(Z)\big) \geq \tau\sum_{i=1}^k\lambda^\ua_i\big(\mathcal{F}(X)\big) + (1-\tau)\sum_{i=1}^k\lambda^\ua_i\big(\mathcal{F}(Y)\big),\quad 1\leq k \leq n.\] 
Let $\bm{a} = \tau\bm{\lambda}\big(\mathcal{F}(X)\big) + (1-\tau)\bm{\lambda}\big(\mathcal{F}(Y)\big)\in \mathbb{R}^n$. Note that for any $1\leq i\leq n$,
\[\big[-\bm{\lambda}\big(\mathcal{F}(Z)\big)\big]^\da_i = -\lambda^\ua_i\big(\mathcal{F}(Z)\big),\quad (-\bm{a})^\da_i = -\tau \lambda^\ua_i\big(\mathcal{F}(X)\big) - (1-\tau)\lambda^\ua_i\big(\mathcal{F}(Y)\big).\]
We thus have $-\bm{\lambda}\big(\mathcal{F}(Z)\big) \prec _w -\bm{a}$. Then by \Cref{lem:MajorBridge} and \Cref{lem:MajorKey}, there exist some $\bm{b}\in \mathbb{R}^n$ and some doubly stochastic matrix $D$ such that $-\bm{\lambda}\big(\mathcal{F}(Z)\big)\leq -\bm{b} = D(-\bm{a})$, or equivalently, 
\[\bm{\lambda}\big(\mathcal{F}(Z)\big)\geq \bm{b} = D\bm{a}.\]
Now for any monotone, concave symmetric form $\phi$ on $\mathbb{R}^n$, we have $\phi\big[\bm{\lambda}\big(\mathcal{F}(Z)\big)\big]\geq \phi(\bm{b})$ due to monotonicity; and $\phi(\bm{b})\geq \phi(\bm{a})$ due to concavity and \Cref{lem:MajorSymmetry}. Also due to concavity of $\phi$ we have
\[\phi(\bm{a}) = \phi\big[\tau\bm{\lambda}\big(\mathcal{F}(X)\big) + (1-\tau)\bm{\lambda}\big(\mathcal{F}(Y)\big)\big]\geq \tau\phi\big[\bm{\lambda}\big(\mathcal{F}(X)\big)\big]+ (1-\tau)\phi\big[\bm{\lambda}\big(\mathcal{F}(Y)\big)\big].\]
Therefore, by the definition we have
\begin{align*}
\phi\big(\mathcal{F}(Z)\big) =&\ \phi\big[\bm{\lambda}\big(\mathcal{F}(Z)\big)\big]\\
\geq&\ \tau\phi\big[\bm{\lambda}\big(\mathcal{F}(X)\big)\big]+ (1-\tau)\phi\big[\bm{\lambda}\big(\mathcal{F}(Y)\big)\big] \\
=&\ \tau\phi\big(\mathcal{F}(X)\big)+ (1-\tau)\phi\big(\mathcal{F}(Y)\big),
\end{align*}
which means $X\mapsto \phi\big(\mathcal{F}(X)\big)$ is concave on $\Omega$.
\end{proof}

\section{Proof of main theorems}
\label{sec:Proofs}
We still need two more lemmas for the proof of our main results. In \Cref{lem:f00}, the infimum is taken over all idempotent matrices of rank $k$. We thus need to use properties of this class of matrices. It is well known that if a matrix is idempotent, then its eigenvalues must be either 0 or 1. Moreover, the following lemma tells that the singular values of a idempotent matrix are either 0, or greater than or equal to 1.
\begin{lemma}\label{lem:Idempotent}
Let $P \in \mathbb{C}^{n\times n}$ be idempotent, i.e. $P^2=P$. Then all non-zero singular values of $P$ are greater than or equal to 1. 
\end{lemma}
\begin{proof}
Let $P=U\Sigma V^*$ be the compact singular value decomposition of $P$, where $U,V\in \mathbb{C}^{n\times k}$ satisfy $U^*U = V^*V =I_k$, $\Sigma \in \mathbf{H}_{++}^k$ is diagonal, and $k=rank(P)$. We need to show that $\Sigma \succeq I_k$. Since $P$ is idempotent, we have
\[U\Sigma V^* = P = P^2 = U\Sigma V^*U\Sigma V^*.\]
Using $U^*U = V^*V =I_k$, we obtain that 
\[\Sigma = \Sigma V^*U\Sigma \ \Longrightarrow\ \Sigma^{-1} = V^*U. \]
Then for any $x\in \mathbb{C}^k$, we have
\[|x^*\Sigma^{-1}x| = |x^*V^*Ux|\leq \|Vx\|_2\|Ux\|_2 = \|x\|_2^2.\]
Therefore $\Sigma^{-1}\preceq I_k$, or equivalently, $\Sigma\succeq I_k$.
\end{proof}

The next lemma is a variation of the Courant-Fisher characterization.

\begin{lemma}\label{lem:PartialCompare}
Let $f:\mathbb{R}\rightarrow \mathbb{R}$ be monotone increasing. Then for any $A\in \mathbf{H}^n$ and any $1\leq k\leq n$,
\begin{equation}\label{eqt:PartialCompare}
\sum_{i=1}^{k}\lambda^\da_i\big(f(A)\big) = \max_{\begin{subarray}{c} Q\in \mathbb{C}^{n\times k} \\ Q^*Q = I_k\end{subarray}} \trace[f(Q^*AQ)], 
\quad \sum_{i=1}^k\lambda^\ua_i(f(A)) = \min_{\begin{subarray}{c} Q\in \mathbb{C}^{n\times k} \\ Q^*Q = I_k\end{subarray}} \trace[f(Q^*AQ)].
\end{equation}
\end{lemma}

\begin{proof}
We only prove the first identity in \Cref{eqt:PartialCompare}. The proof of the second identity is totally parallel. For any $Q\in \mathbb{C}^{n\times k}, Q^*Q = I_k$ and any $V\in \mathbb{C}^{k\times i}, V^*V = I_i$ with $i\leq k\leq n$, we have $QV\in \mathbb{C}^{n\times i}, (QV)^*QV = I_i$. Thus by \Cref{thm:Courant-Fisher} we have
\begin{align*}
\lambda^\da_i(Q^*AQ) = \max_{\begin{subarray}{c} V\in \mathbb{C}^{k\times i} \\ V^*V = I_i\end{subarray}} \min_{\begin{subarray}{c} \bx\in \mathbb{C}^{i} \\ \bx^*\bx=1 \end{subarray}} \trace[\bx^*V^*Q^*AQV\bx] \leq \max_{\begin{subarray}{c} U\in \mathbb{C}^{n\times i} \\ U^*U = I_i\end{subarray}} \min_{\begin{subarray}{c} \bx\in \mathbb{C}^{i} \\ \bx^*\bx=1 \end{subarray}} \trace[\bx^*U^*AU\bx] = \lambda^\da_i(A).
\end{align*}
Since $f$ is monotone increasing, we obtain that 
\[\trace[f(Q^*AQ)] = \sum_{i=1}^k\lambda^\da_i\big(f(Q^*AQ)\big) = \sum_{i=1}^kf\big(\lambda^\da_i(Q^*AQ)\big)\leq \sum_{i=1}^kf\big(\lambda^\da_i(A)\big) = \sum_{i=1}^k\lambda^\da_i\big(f(A)\big).\]
In particular, if we choose $Q=[\bm{q}_1,\dots,\bm{q}_k]\in\mathbb{C}^{n\times k}$ to be the orthonormal eigenvectors of $A$ corresponding to the eigenvalues $\lambda^\da_1(A),\dots,\lambda^\da_k(A)$, we have exactly $\trace[f(Q^*AQ)] = \sum_{i=1}^k\lambda^\da_i\big(f(A)\big)$. Therefore we have
\[\sum_{i=1}^{k}\lambda^\da_i\big(f(A)\big) = \max_{\begin{subarray}{c} Q\in \mathbb{C}^{n\times k} \\ Q^*Q = I_k\end{subarray}} \trace[f(Q^*AQ)].\]
\end{proof}

\begin{proof}[\rm\textbf{Proof of \Cref{lem:f00}}] 
Let $\mathcal{G}_k =\{G\in\mathbb{C}^{n\times n}:G^2=G,\mathrm{rank}(G)=k\}$. We first prove identity \eqref{eqt:f00} with ``$\inf$'' replaced by ``$\min$'', for any invertible $M$. We need to show that the inequality 
\begin{equation}\label{eqt:ineqt1}
\sum_{i=1}^k \lambda^\ua_i\big(f(M^*AM)\big) \leq \trace\big[f(M^*G^*AGM)\big]
\end{equation}
holds for any $G\in\mathcal{G}_k$. We define 
\[P = M^{-1}GM.\]
Since $G^2 = G$, we have $P^2 = M^{-1}GMM^{-1}GM = M^{-1}G^2M = P$. That is, P is idempotent. Also we have $\mathrm{rank}(P) = \mathrm{rank}(G)=k$. Let $P=U\Sigma V^*$ be the compact singular value decomposition of $P$, where $U,V\in \mathbb{C}^{n\times k}$ satisfy $U^*U = V^*V =I_k$, and $\Sigma \in \mathbf{H}_{++}^k$ is diagonal. By \Cref{lem:Idempotent}, we know $\Sigma \succeq I_k$. Then we have
\begin{align*}
\trace\big[f(M^*G^*AGM)\big] =&\ \trace\big[f(P^*M^*AMP)\big]\\
=&\ \trace\big[f(V\Sigma U^*M^*AMU\Sigma V^*)\big]\\
=&\ \trace\big[f(A^\frac{1}{2}MU\Sigma V^*V\Sigma U^*M^*A^\frac{1}{2})\big]\\
=&\ \trace\big[f(A^\frac{1}{2}MU\Sigma^2U^*M^*A^\frac{1}{2})\big]
\end{align*}
We have used the fact that $\trace[f(X^*X)]=\trace[f(XX^*)]$ for any $X\in \mathbb{C}^{n\times m}$, since the spectrum of $X^*X$ and the spectrum of $XX^*$ may only differ by some zeros, but we have $f(0)=0$. Since $\Sigma\succeq I_k$ and $\Sigma$ is diagonal, we have $\Sigma^2\succeq I_k$, and thus 
\[A^\frac{1}{2}MU\Sigma^2U^*M^*A^\frac{1}{2}\succeq A^\frac{1}{2}MUU^*M^*A^\frac{1}{2}.\]
Since $f$ is monotone increasing, by \Cref{prop:Composition}, $\trace[f(\cdot)]$ is a monotone symmetric form. Therefore we obtain
\[\trace\big[f(A^\frac{1}{2}MU\Sigma^2U^*M^*A^\frac{1}{2})\big]\geq \trace\big[f(A^\frac{1}{2}MUU^*M^*A^\frac{1}{2})\big] = \trace\big[f(U^*M^*AMU)\big].\]
Again since $f$ is monotone increasing, by \Cref{lem:PartialCompare}, we have
\[\trace\big[f(U^*M^*AMU)\big]\geq \sum_{i=1}^k\lambda^\ua_i\big(f(M^*AM)\big).\]
So we have proved inequality \eqref{eqt:ineqt1}. We then need to find some $G\in \mathcal{G}_k$ so that the equality in \eqref{eqt:ineqt1} holds. In fact, we can choose $G = MQQ^*M^{-1}$, where $Q=[\bm{q}_1,\dots,\bm{q}_k]\in \mathbb{C}^{n\times k},Q^*Q=I_k$ and $\bm{q}_i$ is the normalized eigenvector of $M^*AM$ corresponding to the eigenvalue $\lambda^\ua_i(M^*AM)$. It is easy to see that $\mathrm{rank}(G)=k$ and $G^2 = MQQ^*M^{-1}MQQ^*M^{-1} = MQQ^*M^{-1} = G$. Moreover, we have
\[M^*G^*AGM = M^*(M^*)^{-1}QQ^*M^*AMQQ^*M^{-1}M = QQ^*M^*AMQQ^*,\]
and thus
\begin{align*}
\trace\big[f(M^*G^*AGM)\big] =&\ \trace\big[f(QQ^*M^*AMQQ^*)\big] \\
=&\ \trace\big[f(Q^*M^*AMQ)\big] \\
=&\ \sum_{i=1}^k \lambda^\ua_i\big(f(M^*AM)\big).
\end{align*}

Next, we will prove identity \eqref{eqt:f00} for a general $M$ that is not necessarily invertible. For any $M\in \mathbb{C}^{n\times n}$, we can always find a sequence $\{M_j\}_{j=1}^{+\infty}\subset\mathbb{C}^{n\times n}$ such that (i) $M_j\rightarrow M$ entry-wisely as $j\rightarrow +\infty$, (ii) each $M_j$ is invertible, and (iii) $M_jM_j^*\succeq MM^*$. Such sequence $\{M_j\}_{j=1}^{+\infty}$ can be easily obtained by only modifying the singular values of $M$. Note that $f$ is continuous since it is convex; ordered eigenvalues and trace are also continuous on $\mathbf{H}^n$. Therefore, for any $G\in\mathcal{G}_k$, we have
\begin{align*}
\sum_{i=1}^k \lambda^\ua_i\big(f(M^*AM)\big) =&\ \lim_{j\rightarrow +\infty}\sum_{i=1}^k \lambda^\ua_i\big(f(M_j^*AM_j)\big) \\
\leq&\  \lim_{j\rightarrow +\infty}\trace\big[f(M_j^*G^*AGM_j)\big] \\
=&\ \trace\big[f(M^*G^*AGM)\big].
\end{align*}
Moreover, for each $M_j$, there is some $G_j\in \mathcal{G}_k$ such that $\trace\big[f(M_j^*G_j^*AG_jM_j)\big] = \sum_{i=1}^k \lambda^\ua_i\big(f(M_j^*AM_j)\big)$. 
Thus we have
\begin{align*}
\sum_{i=1}^k \lambda^\ua_i\big(f(M^*AM)\big) \leq&\ \trace\big[f(M^*G_j^*AG_jM)\big]\\
=&\ \trace\big[f(A^\frac{1}{2}G_jMM^*G_j^*A^\frac{1}{2})\big]\\
\leq&\ \trace\big[f(A^\frac{1}{2}G_jM_jM_j^*G_j^*A^\frac{1}{2})\big]\\
=&\ \trace\big[f(M_j^*G_j^*AG_jM_j)\big]\\
=&\ \sum_{i=1}^k \lambda^\ua_i\big(f(M_j^*AM_j)\big).
\end{align*}
Then again since $\sum_{i=1}^k \lambda^\ua_i\big(f(M^*AM)\big) = \lim_{j\rightarrow +\infty}\sum_{i=1}^k \lambda^\ua_i\big(f(M_j^*AM_j)\big)$, we must have
\[\sum_{i=1}^k \lambda^\ua_i\big(f(M^*AM)\big) = \lim_{j\rightarrow +\infty} \trace\big[f(M^*G_j^*AG_jM)\big],\]
and so identity \eqref{eqt:f00} is proved.
\end{proof}

\begin{proof}[\rm\textbf{Proof of \Cref{lem:fNI0}}] 
Since $f$ is monotone increasing on $\mathbb{R}$, we have $f(x)\geq f(-\infty) = 0$ for all $x\in \mathbb{R}$, and thus $f(X)\in\mathbf{H}_+^n$ for all $X\in \mathbf{H}^n$. Let $\mathcal{H}_k = \{H\in \mathbf{H}^n:\mathrm{rank}(H)=n-k\}$. For any $H\in \mathcal{H}_k$, since $\dim\mathrm{Null}(H) = k$, we can always find some $U\in \mathbb{C}^{n\times k}$ such that $U^*U = I_k$ and $HU = \bm{0}$. Then by \Cref{lem:PartialCompare} we have
\[\trace\big[f(H+A)\big]\geq \sum_{i=1}^k \lambda^\da_i\big(f(H+A)\big)\geq \trace[f(U^*(H+A)U)] = \trace[f(U^*AU)]\geq \sum_{i=1}^k\lambda^\ua_i\big(f(A)\big).\] 
Next we need to show that for arbitrary small $\epsilon>0$, there is some $H_\delta \in \mathcal{H}_k$ such that 
\[\trace\big[f(H_\delta+A)\big]\leq \sum_{i=1}^k\lambda^\ua_i\big(f(A)\big) + \epsilon.\]
Let $A = Q\Lambda Q^*$ be the an eigenvalue decomposition of $A$, where $Q\in \mathbb{R}^{n\times n}$ is unitary, and $\Lambda$ is diagonal with ascending diagonal entries $\lambda^\ua_1(A),\dots,\lambda^\ua_n(A)$. We then take $H_\delta = Q\Lambda_\delta Q^*$, where $\Lambda_\delta$ is also diagonal, and the $i_{\text{th}}$ diagonal entry of $\Lambda_\delta$ is 
\[0\quad \text{if}\quad i\leq k;\quad \text{or}\quad -\delta - \lambda^\ua_i(A)\quad \text{if}\quad i>k.\]
When $\delta$ is large enough, we can have $-\delta - \lambda^\ua_i(A)<0$ for all $k<i\leq n$, and thus $H_\delta\in \mathcal{H}_k$. And we have
\[\trace\big[f(H_\delta+A)\big] = \trace\big[f(Q(\Lambda_\delta+\Lambda)Q^*)\big] =  \trace\big[f(\Lambda_\epsilon+\Lambda)\big] = (n-k)f(-\delta)+\sum_{i=1}^k\lambda^\ua_i\big(f(A)\big).\] 
Since $f(x\rightarrow -\infty) = 0$, we can always choose $\delta$ large enough so that $H_\delta\in \mathcal{H}_k$ and $(n-k)f(-\delta)\leq \epsilon$. So we have proved identity \eqref{eqt:fNI0}.
\end{proof}

\begin{proof}[\rm\textbf{Proof of \Cref{thm:GeneralLiebConcavity}}] 
We only need to show that 
\begin{equation}\label{eqt:lambdafunction1}
(A,B) \ \longmapsto\ \sum_{i=1}^k\lambda^\ua_i\big((B^\frac{qs}{2}K^*A^{ps}KB^\frac{qs}{2})^{\frac{1}{s}}\big) 
\end{equation}
is jointly concave on $\mathbf{H}_+^m\times\mathbf{H}_+^n$ for all $1\leq k\leq n$. According to Theorem 3.2 in \cite{huang2019generalizing}, for any $L\in \mathbb{C}^{m\times n}$, the function 
\[(A,B) \ \longmapsto\ \trace\big[(B^\frac{qs}{2}L^*A^{ps}LB^\frac{qs}{2})^{\frac{1}{s}}\big] \]
is jointly concave on $\mathbf{H}_+^m\times\mathbf{H}_+^n$. Thus for any $A_1,B_1\in \mathbf{H}_+^m,A_2,B_2\in \mathbf{H}_+^n$ and any $\tau \in [0,1]$, with $C_i=\tau A_i+(1-\tau)B_i,i=1,2$, we have
\begin{align*}
\sum_{i=1}^k\lambda^\ua_i\big((C_2^\frac{qs}{2}K^*C_1^{ps}KC_2^\frac{qs}{2})^{\frac{1}{s}}\big) =&\ \inf_{\begin{subarray}{c}G\in\mathbb{C}^{n\times n},G^2=G\\\mathrm{rank}(G)=k\end{subarray}} \trace\big[(C_2^\frac{qs}{2}G^*K^*C_1^{ps}KGC_2^\frac{qs}{2})^{\frac{1}{s}}\big]\\
\geq &\ \inf_{\begin{subarray}{c}G\in\mathbb{C}^{n\times n},G^2=G\\\mathrm{rank}(G)=k\end{subarray}} \Big\{\tau\trace\big[(A_2^\frac{qs}{2}G^*K^*A_1^{ps}KGA_2^\frac{qs}{2})^{\frac{1}{s}}\big] \\
&\ \qquad\qquad\qquad\quad+(1-\tau)\trace\big[(B_2^\frac{qs}{2}G^*K^*B_1^{ps}KGB_2^\frac{qs}{2})^{\frac{1}{s}}\big]\Big\}\\
\geq &\ \tau \inf_{\begin{subarray}{c}G\in\mathbb{C}^{n\times n},G^2=G\\\mathrm{rank}(G)=k\end{subarray}} \trace\big[(A_2^\frac{qs}{2}G^*K^*A_1^{ps}KGA_2^\frac{qs}{2})^{\frac{1}{s}}\big] \\
&\ + (1-\tau) \inf_{\begin{subarray}{c}G\in\mathbb{C}^{n\times n},G^2=G\\\mathrm{rank}(G)=k\end{subarray}} \trace\big[(B_2^\frac{qs}{2}G^*K^*B_1^{ps}KGB_2^\frac{qs}{2})^{\frac{1}{s}}\big] \\
=&\ \tau \sum_{i=1}^k\lambda^\ua_i\big((A_2^\frac{qs}{2}K^*A_1^{ps}KA_2^\frac{qs}{2})^{\frac{1}{s}}\big) + (1-\tau) \sum_{i=1}^k\lambda^\ua_i\big((B_2^\frac{qs}{2}K^*B_1^{ps}KB_2^\frac{qs}{2})^{\frac{1}{s}}\big).
\end{align*}
We have used formula \eqref{eqt:f00} from \Cref{lem:f00} with $f(x) = x^\frac{1}{s}$, which is monotone increasing on $\mathbb{R}_+$ and satisfies $f(0)=0$. So we have proved the concavity of \eqref{eqt:lambdafunction1} for all $1\leq k\leq n$. The concavity of \eqref{eqt:function1} then follows from \Cref{thm:GenericConcavity} with 
\[\mathcal{F}:\mathbf{H}_+^m\times\mathbf{H}_+^n\longrightarrow \mathbf{H}_+^n,\quad \mathcal{F}(A,B) = (B^\frac{qs}{2}K^*A^{ps}KB^\frac{qs}{2})^{\frac{1}{s}}.\]
\end{proof}

\begin{proof}[\rm\textbf{Proof of \Cref{thm:GeneralLiebConcavity}}] 
We only need to show that 
\begin{equation}\label{eqt:lambdafunction2}
(A_1,A_2,\dots,A_m) \ \longmapsto\ \sum_{i=1}^k\lambda^\ua_i\big(\exp\big(H+\sum_{j=1}^mp_j\log A_j\big)\big) 
\end{equation}
is jointly concave on $(\mathbf{H}_{++}^n)^{\times m}$ for all $1\leq k\leq n$. According to Corollary 6.1 in \cite{LIEB1973267}(see also Theorem 3.3 in \cite{huang2019generalizing}), for any $L\in \mathbf{H}^n$, the function 
\[(A_1,A_2,\dots,A_m) \ \longmapsto\ \trace\big[\exp\big(L+\sum_{j=1}^mp_j\log A_j\big)\big]  \]
is jointly concave on $(\mathbf{H}_{++}^n)^{\times m}$. Thus for any $(A_1,A_2,\dots,A_m),(B_1,B_2,\dots,B_m)\in (\mathbf{H}_{++}^n)^{\times m}$ and any $\tau \in [0,1]$, with $C_i=\tau A_i+(1-\tau)B_i,i=1,\dots,m$, we have
\begin{align*}
\sum_{i=1}^k\lambda^\ua_i\big(\exp\big(H+\sum_{j=1}^mp_j\log C_j\big)\big)  =&\ \inf_{\begin{subarray}{c}M\in\mathbf{H}^n\\\mathrm{rank}(M)=n-k\end{subarray}} \trace\big[\exp\big(M+H+\sum_{j=1}^mp_j\log C_j\big)\big]\\
\geq &\ \inf_{\begin{subarray}{c}M\in\mathbf{H}^n\\\mathrm{rank}(M)=n-k\end{subarray}} \Big\{\tau\trace\big[\exp\big(M+H+\sum_{j=1}^mp_j\log A_j\big)\big] \\
&\ \qquad\qquad\qquad\quad+(1-\tau)\trace\big[\exp\big(M+H+\sum_{j=1}^mp_j\log B_j\big)\big]\Big\}\\
\geq &\ \tau \inf_{\begin{subarray}{c}M\in\mathbf{H}^n\\\mathrm{rank}(M)=n-k\end{subarray}} \trace\big[\exp\big(M+H+\sum_{j=1}^mp_j\log A_j\big)\big] \\
&\ + (1-\tau) \inf_{\begin{subarray}{c}M\in\mathbf{H}^n\\\mathrm{rank}(M)=n-k\end{subarray}} \trace\big[\exp\big(M+H+\sum_{j=1}^mp_j\log B_j\big)\big] \\
=&\ \tau \sum_{i=1}^k\lambda^\ua_i\big(\exp\big(H+\sum_{j=1}^mp_j\log A_j\big)\big) \\
&\ + (1-\tau) \sum_{i=1}^k\lambda^\ua_i\big(\exp\big(H+\sum_{j=1}^mp_j\log B_j\big)\big).
\end{align*}
We have used formula \eqref{eqt:fNI0} from \Cref{lem:fNI0} with $f(x) = \exp(x)$, which is monotone increasing on $\mathbb{R}$ and satisfies $f(x\rightarrow-\infty)=0$. So we have proved the concavity of \eqref{eqt:lambdafunction2} for all $1\leq k\leq n$. The concavity of \eqref{eqt:function2} then follows from \Cref{thm:GenericConcavity} with 
\[\mathcal{F}:(\mathbf{H}_{++}^n)^{\times m}\longmapsto \mathbf{H}_{++}^n,\quad \mathcal{F}(A_1,A_2,\dots,A_m) = \exp\big(H+\sum_{j=1}^mp_j\log A_j\big).\]
\end{proof}

\section*{Acknowledgment}
The research was in part supported by the NSF Grant DMS-1613861. The author would like to thank Thomas Y. Hou for his wholehearted mentoring and supporting.

\newpage
\bibliographystyle{elsarticle-num}
\bibliography{reference}

\end{document}